 \documentclass[graphicx,amscd,amssymb,verbatim,11pt,righttag,color]{amsart}
\usepackage{graphicx}
\usepackage{enumerate,color}
\usepackage{arydshln,mathrsfs}

\newtheorem{theorem}{Theorem}[section]
\newtheorem{lemma}[theorem]{Lemma}
\newtheorem{conjecture}[theorem]{Conjecture}
\newtheorem{proposition}[theorem]{Proposition}

\newtheorem{definition}{Definition}[section]
\newtheorem{remark}{Remark}[section]

\newtheorem{example}[theorem]{Example}
\numberwithin{equation}{section}

\begin{document}

\title  {Translational absolute continuity and  Fourier frames on a sum of singular measures}


\author{Xiaoye Fu}

\address{ School of Mathematics and Statistics, The Central China Normal University, Wuhan, China}

 \email{xiaoyefu@mail.ccnu.edu.cn}

\author{Chun-Kit Lai}

\address{Department of Mathematics, San Francisco State University,
1600 Holloway Avenue, San Francisco, CA 94132.}

 \email{cklai@sfsu.edu}

\thanks{ The research is supported in part by the NSFC 11401205}

\subjclass[2010]{Primary 28A25, 42A85, 42B05.}
\keywords{Fourier frame, frame-spectral measures, packing pair, self-affine measures}

\begin{abstract}
A finite Borel measure $\mu$ in ${\mathbb R}^d$ is called a frame-spectral measure if it admits an exponential frame (or Fourier frame) for $L^2(\mu)$. It has been conjectured that a frame-spectral measure must be  translationally absolutely continuous, which is a criterion describing the local uniformity of a measure on its support. In this paper, we show that if any measures $\nu$ and $\lambda$  without atoms whose supports form a packing pair, then $\nu\ast \lambda +\delta_t\ast\nu$  is translationally singular and it does not admit any Fourier frame. In particular, we show that the sum of one-fourth and one-sixteenth Cantor measure $\mu_4+\mu_{16}$ does not admit any Fourier frame.
We also interpolate the mixed-type frame-spectral measures studied by Lev and the measure we studied. In doing so, we demonstrate a discontinuity behavior: For any anticlockwise rotation mapping $R_{\theta}$ with $\theta\ne \pm\pi/2$, the two-dimensional measure $\rho_{\theta} (\cdot): = (\mu_4\times\delta_0)(\cdot)+(\delta_0\times\mu_{16})(R_{\theta}^{-1}\cdot)$, supported on the union of $x$-axis and $y=(\cot \theta)x$,  always admit a Fourier frame. Furthermore, we can find $\{e^{2\pi i \langle\lambda,x\rangle}\}_{\lambda\in\Lambda_{\theta}}$ such that it forms a Fourier frame for $\rho_{\theta}$ with frame bounds independent of $\theta$. Nonetheless, $\rho_{\pm\pi/2}$ does not admit any Fourier frame.
\end{abstract}

\renewcommand{\baselinestretch}{1.0}
\maketitle

\vspace{-1.2cm}
\section{Introduction}

Let $\mu$ be a finite Borel measure on $\mathbb{R}^d$ and $\langle\cdot,\cdot\rangle$ denote the standard inner product. We say that $\mu$ is a \emph{frame-spectral measure} if there exists a countable set $\Lambda\subset\mathbb{R}^d$ such that a system of exponential functions $E(\Lambda): = \{e_{\lambda}(x): \lambda\in\Lambda\}$, where $e_{\lambda}(x)=e^{2\pi i\langle\lambda,x\rangle}$, forms a \emph{Fourier frame} for $L^2(\mu)$, i.e., there exist two constants $A,B$ such that $0 < A \le B<\infty$ with
\begin{eqnarray}\label{eqframe}
A\|f\|_{L^2(\mu)}^2 \le \sum\limits_{\lambda\in \Lambda} |\int f(x)e^{2\pi i \langle\lambda,x\rangle}d\mu(x)|^2 \le B\|f\|_{L^2(\mu)}^2
\end{eqnarray}
 for any $f\in L^2(\mu)$. Whenever such $\Lambda$ exists, $\Lambda$ is called a \emph{frame spectrum} for $\mu$. If only the upper bound holds, $E(\Lambda)$ is called \emph{a Bessel sequence} for $\mu$. We say that $\mu$ is a \emph{spectral measure} and $\Lambda$ is a \emph{spectrum} for $\mu$ if the system $E(\Lambda)$ forms an orthonormal basis for $L^2(\mu)$.

\vspace{0.3cm}

The notion of Fourier Frames is a natural generalization of an exponential orthonormal basis, which is a spanning set with possible redundancies in the representation. It was first introduced by Duffin and Schaeffer \cite{DS52} in the context of non-harmonic Fourier series. Frames nowadays have wide range of applications in signal transmission and reconstruction. The theory of Fourier frames have also been studied in complex analysis since it is equivalent to having a stable sampling in Paley-Wiener spaces \cite{OS02}. For more general background of frame theory, readers may refer to \cite{Chr}.

\vspace{0.3cm}

One of the basic questions in frame theory is to classify when a  measure $\mu$ is  spectral or frame-spectral. The origin of this question dates back to Fuglede \cite{Fug74} who initiated a study to determine when a measurable set $\Omega$ admits an exponential orthonormal basis for $L^2(\Omega)$ and proposed his famous spectral set conjecture. Although the conjecture was disproved in ${\mathbb R}^d, d\ge 3$ by Tao, Kolountzakis and Matolcsi \cite{T04, KM1, KM2}, the conjecture has led people to research extensively the property leading a set or a measure to be spectral or frame-spectral. Particularly, the study of singular spectral measures is possible, difficult but compelling  \cite{JP98,LW02, DJ07, S00}. One of the major properties of frame-spectral measures that has been observed in several earlier work (\cite{LW06,DHJ09}) and was rigorously formulated in \cite{DL14}, is a notion of ``translational uniformity".

\vspace{0.3cm}

Given a finite Borel measure $\mu$,  the {\it closed support} of $\mu$ is the smallest closed set $K$ such that $\mu$ has full measure (i.e. $\mu(K)  = \mu({\mathbb R}^d)$). We will denote it by $K_{\mu}$ throughout the paper. We say that a Borel set $X$ is a {\it Borel support} of $\mu$ if $X\subset K_{\mu}$ and $\mu(X) = \mu(K_{\mu})$. Note that the closed support of $\mu$ is unique, but Borel support is not unique.  We have the following definition, which refines the definition proposed by   Dutkay and the second-named author:

\begin{definition}
{\rm Let $\mu$ be a finite Borel measure and let $E$ be a set  such that $E\subset K_\mu$ and $\mu(E)>0$. We say that $\mu$ is {\it translationally absolutely continuous on $E$} if for any $t\in{\mathbb R}^d$, the measure}
$$
\omega_t(F): = \mu ((F+t)\cap (E+t))
$$
{\rm is absolutely continuous with respect to $\mu$. We say that $\mu$ is {\it translationally absolutely continuous} if there exists a Borel support $X$ of $\mu$ such that  it is translationally absolutely continuous on every set $E$  with $E\subset X$ and  $\mu(E)>0$. We will also say that $\mu$ is {\it translationally singular} if $\mu$ is not  translationally absolutely continuous.}
\end{definition}

The following are some remark concerning the property of translational absolute continuity.

\begin{remark}\label{remark}  {\rm From the definition of translational absolute continuity, we have:}
\begin{enumerate}[(1)]
\item {\rm In the definition, if $\mu(E+t) = 0$, then $\omega_t = 0$ and it is trivially absolutely continuous
 with respect to $\mu$. Therefore, it suffices to check those $t$ such that $\mu(E+t)>0$. }
\item {\rm Hausdorff measure with any gauge functions restricted on a set $K$ must be  translationally absolutely continuous, since it is translational invariant.}
\item {\rm It is possible that $\mu$ is not translationally absolutely continuous on $K_{\mu}$, but it is translationally absolutely continuous on some Borel support $X$ of $\mu$ (See Example \ref{example5}). Therefore, the definition above needs this finer notion of Borel support. }
\item {\rm Translational absolute continuity can be illustrated as some common language: On $E$ in Figure 1, we may imagine the length is measured by kilometers, while on $E+t$, the length is measured in miles. $\omega_t$ measures $F$ in $E$ using the translation $F+t$, which means that the length is now measured using miles in $E$. This action is clearly an absolutely continuous measurement.}
\begin{figure}[h]\label{fig}
  \includegraphics[width=8cm]{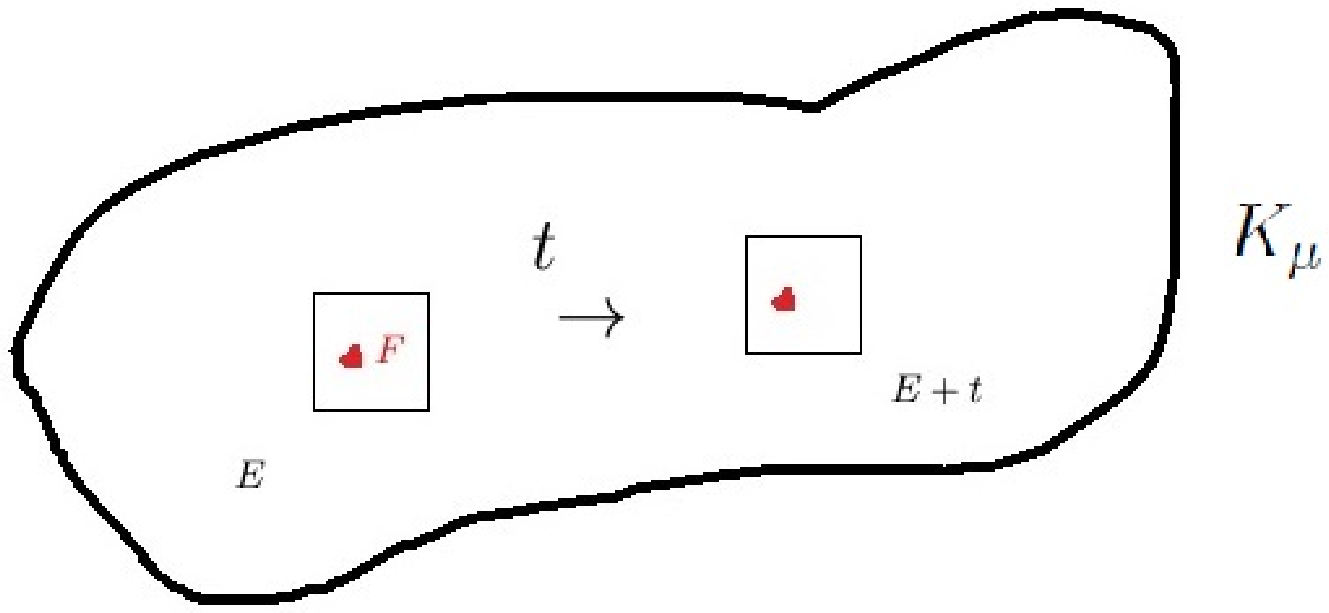}\\
  \caption{}
\end{figure}
\end{enumerate}
\end{remark}

\vspace{0.3cm}

The upper bound of following theorem was proved in \cite{DL14}.

\begin{theorem}\label{th1}
Let $\mu$ be a finite Borel measure which admits some Fourier frame $\{e_{\lambda}(x)\}_{ \lambda\in\Lambda}$ with frame bounds $0<A\le B<\infty$. Suppose that $\mu$ is translationally absolutely continuous on $E\subset K_\mu$ with $\mu(E)>0$. Then the essential supremum of the Radon-Nikodym derivative of $\omega_t$ with respect to $\mu$ on $E$, $\frac{d\omega_t}{d\mu}$, satisfies
\begin{equation}\label{eq1}
 \left\|\frac{d\omega_t}{d\mu}\right\|_{\infty}\le \frac{B}{A}.
\end{equation}
\end{theorem}

This theorem suggests a frame-spectral measure, if  it is translationally absolutely continuous, must be uniform along the support under translation in the sense that the density cannot vary too much as in (\ref{eq1}). Using this result, Dutkay and Lai proved an absolutely continuous  measure (with respect to Lebesgue measure) must have density bounded above and away from zero on the support (it was proved earlier in \cite{Lai11} using Beurling density). Furthermore, if the measure admits a tight Fourier frame ($A= B$ in (\ref{eqframe})), then the density is a constant, generalizing fully a theorem in \cite{LW06}. They also proved that a self-similar measure under a measure-disjoint condition is translationally absolutely continuous, and if it is frame-spectral, then it must be of equal-weight on each map. This implies that the measure is essentially a Hausdorff measure of a certain gauge function. Because of its wide applicability and examples of measures with Fourier frames are just some Hausdorff measures, The following conjecture was therefore proposed:

\begin{conjecture}\label{conj1}
If $\mu$ is a frame-spectral measure, then $\mu$ must be translationally
absolutely continuous.
\end{conjecture}

\vspace{0.3cm}

We note that translational continuity does not imply  a strong condition that the measure is absolutely continuous with respect to a certain Hausdorff measure as in (2) of Remark \ref{remark}. Recently, Lev \cite{L16} proved

\vspace{0.3cm}

\begin{theorem}[\cite{L16}]\label{t-1-3}
Let $\mu$ and $\nu$ be positive and finite continuous measures (without atoms) on ${\mathbb R}^m$ and ${\mathbb R}^{d-m}$ respectively. Suppose that $\mu$ and $\nu$ are frame-spectral measures. Then the measure $\rho: = \mu\times \delta_{\bf 0^{d-m}} + \delta_{\bf 0^{m}}\times\nu$ is a frame-spectral measure on ${\mathbb R}^d$.
\end{theorem}

This theorem implies that a frame-spectral measure can have component of different Hausdorff dimensions, so there is no absolute continuity with respect to any Hausdorff measure.  Nonetheless, since the subspace ${\mathbb R}^m\times\{\bf 0^{d-m}\}$ and $\{\bf 0^m\}\times{\mathbb R}^{d-m}$ intersects only at the origin, the only set   from the support of  $\mu\times \delta_{\bf 0^{d-m}}$ such that it can be translated into  the support  of $\delta_{\bf 0^{d-m}}\times \nu$ is just a singleton. But these measures have no atoms. Thus, $\rho$ is still translationally absolutely continuous (provided that $\mu$ and $\nu$ is) and conjecture \ref{conj1} is not violated.

\vspace{0.3cm}

On the other hand, we can construct easily measures that are not translationally absolutely continuous using some known examples and their sum. Our result in this paper is to show that these measures are not frame-spectral, suggesting a positive evidence towards Conjecture \ref{conj1} and complementing the result of Lev about sum of two frame-spectral measures.

\vspace{0.3cm}

 Recall that Jorgensen and Pedersen \cite{JP98}, who discovered the first  singular fractal spectral measures, found that the  standard Cantor measure with contraction ratio $1/2n$, $n\in{\mathbb N}$, is  a spectral measure, while the standard Cantor measure with contraction ratio $1/(2n+1)$ is not. For further investigation of spectral measures associated to different iterated function systems, one can refer to \cite{D12, DHL13, DHL14, DHL}.  If we let $\mu_4$ be the standard one-fourth Cantor measure with the digit set $\{0,1\}$ and $\mu_{16}$ be the standard one-sixteenth Cantor measure with the digit set $\{0,1\}$. i.e. they are the unique self-similar measure satisfying
$$
\mu_4(E) = \frac{1}{2}\mu_4( 4E)+\frac{1}{2}\mu_4(4E-1), \
\mu_{16}(E) = \frac{1}{2}\mu_{16}( 16E)+\frac{1}{2}\mu_{16}(16E-1).
$$
It is well-known that the supports of $\mu_4$ and $\mu_{16}$ are respectively
$$
K_{\mu_4}= \left\{\sum_{k=1}^{\infty}\frac{\epsilon_j}{4^j}: \epsilon_j\in\{0,1\}\right\}, \ K_{\mu_{16}} = \left\{\sum_{k=1}^{\infty}\frac{\epsilon_j}{(16)^j}: \epsilon_j\in\{0,1\}\right\}.
$$
Let $x\in{\mathbb R}$ and define
$$
\rho = \rho_x  = \mu_4+\delta_x\ast\mu_{16},
$$
where $\delta_x\ast \mu (E): = \mu(E-x)$. This measure is supported on $K_{\mu_4} \cup (x+K_{\mu_{16}})$ and we now claim that it is translationally singular if $K_{\mu_4}\cap (x+K_{\mu_{16}}) = \emptyset$ (this assumption will be removed later). Indeed, we note that $K_{\mu_{16}}\subset K_{\mu_4}$. We take $E = K_{\mu_4}$, then $\rho(E) \ge \mu_4(E) = 1>0$, we consider $F = K_{\mu_{16}}$ and take $t=x$, then
 $$
\rho(F) = \mu_4(F)= 0, \ \mbox{but} \ \omega_t(F) = \rho(K_{16}+t) =\mu_{16}(F) = 1.
 $$
By the standard definition of absolute continuity \cite{Rud87}, $\omega_t$ is singular with respect to $\rho$ and thus $\rho$ is translationally singular on $E$.

\medskip

We will introduce a notion of packing pair (See definition is Section 2) for any two compact sets (with Lebesgue measure equals zero probably) and explore many examples. Our main result is the following theorem towards the relations of packing pair with frame-spectrality and translational singularity:

\begin{theorem}\label{t-1-1}
Let $(K_1,K_2)$ be a packing pair and $\nu,\lambda$ be two continuous measures without atoms supported on $K_1,K_2$ respectively. Let also $\mu = \nu\ast\lambda$. Then for any $t\in{\mathbb R}^d$, the measure
$$
\rho_t = \mu+\delta_t\ast\nu
$$
is translational singular and does not admit any Fourier frame.
\end{theorem}
%
%

%
%
%
%
%
%
%
%
\vspace{0.3cm}

 Note that $\mu_4 = \mu_{16}\ast\lambda_{16}$ for some measure $\lambda_{16}$ with $K_{\lambda_{16}}$ and $K_{\mu_{16}}$ forms a packing pair (see Proposition \ref{prop3.5}), the above theorem tells us that $\rho_0 = \mu_4+\mu_{16}$ is not frame spectral.  However, using the result of Lev (Theorem \ref{t-1-3}), we know that if we place $\mu_4$ and $\mu_{16}$ in ${\mathbb R}^2$ in the way that
\begin{equation}\label{eq1.2}
\rho = \mu_4\times\{0\} +\{0\}\times\mu_{16},
\end{equation}
$\rho$ is a frame-spectral measure on ${\mathbb R}^2$. In the following, we will interpolate the measure in (\ref{eq1.2}) and $\mu_4+\mu_{16}$ through a rotation.

\vspace{0.3cm}

Let $\mu$ and $\nu$ be defined as in Theorem \ref{t-1-3}.  Let $V$ be a subspace of dimension $d$ and let $T: {\mathbb R}^d \rightarrow {\mathbb R}^d$ be an invertible linear transformation such that
\begin{eqnarray*}
T(\{{\bf 0}^m\} \times {\mathbb R}^{d-m}) = V.
\end{eqnarray*}
Define $\nu_T(E) = (\delta_{\bf 0^{m}}\times \nu)(T^{-1}(E))$ for any Borel set $E\subset{\mathbb R}^d$. This measure will be supported inside $V$. We show that

\begin{theorem}\label{t-1-4}
Let $\mu$ and $\nu$ be defined as in Theorem \ref{t-1-3}. Suppose that $V\cap ({\mathbb R}^m\times \{{\bf 0}^{d-m}\}) = \{{\bf 0}^d\}$. Then $\rho_T: = \mu\times\{{\bf 0}^{d-m}\} + \nu_T$ is a frame-spectral measure. Moreover, there exists $0<A\le B<\infty$ such that for any invertible linear transformation $T$ satisfying $V\cap ({\mathbb R}^m\times \{{\bf 0}^{d-m}\}) = \{{\bf 0}^d\}$, there exists a frame spectrum $\Lambda_T$ such that $E(\Lambda_T)$ is a Fourier frame for $\rho_T$ with frame bounds $A,B$.
\end{theorem}

\vspace{0.3cm}

Combining Theorem \ref{t-1-4} and Theorem \ref{t-1-1}, an interesting discontinuity phenomenon is observed.

\begin{theorem}
Let $R_{\theta}$ be a rotation matrix on ${\mathbb R}^2$ by an angle $\theta$. The measure
$$
\rho(E) = (\mu_4\times \{0\})(E)+  (\{0\}\times \mu_{16}) (R_{\theta}^{-1}(E))
$$
are frame-spectral measures as long as $\theta \ne \pm\pi/2$. When $\theta =\pm\pi/2$, Theorem \ref{t-1-1} shows that they are not frame-spectral. Furthermore, if $\theta\ne \pm\pi/2$, we can always find some $E(\Lambda_{\theta})$ forming a Fourier frame with certain fixed frame bounds.
\end{theorem}
\vspace{0.3cm}

In the end of the introduction, we remark that singular measures with Fourier frames but not an exponential basis were studied in  \cite{DHSW11, DHW11, DHW14,HLL,LW}. In particular, in \cite{LW}, such measures were shown to exist in abundance by relaxing the Hadamard triple condition proposed in \cite{JP98} and \cite{S00}.

\vspace{0.3cm}

We organize the paper as follows. In Section 2, we introduce some notations and properties on frame measures and packing pairs. Theorem \ref{t-1-1} (Theorem \ref{th2.1}) will be proved. In Section 3, We give some examples of packing pairs and $\mu_4$ and $\mu_{16}$ will be obtained as a special case of our examples. In Section 4, Theorem \ref{t-1-4} will be proved. In Section 5, we conclude with some remarks and follow-up problems related to Fourier frames for a sum of singular measures.

\vspace{0.3cm}

\section {Frame measures and Packing pairs}

\subsection{Frame measures.}
Let $\mu$ be a finite Borel measure. Given $f\in L^1(\mu)$,  We define the Fourier transform with respect to $\mu$ to be
 $$
(fd\mu)^{\widehat{}}(\xi) = \int f(x)e^{-2\pi i \langle\xi,x\rangle }d\mu(x).
$$
By putting $f$ to be the characteristic function on the support of $\mu$, we obtain the Fourier transform of a measure $\mu$
$$
\widehat{\mu}(\xi) =  \int e^{-2\pi i \langle\xi,x\rangle }d\mu(x).
$$
We say that a locally finite measure $\Gamma$ is a {\it frame measure} of $\mu$ if there exists $0<A\le B<\infty$ such that
\begin{equation}\label{eqframemeasure}
A\|f\|_{L^2(\mu)}^2 \le \int \left| (fd\mu)^{\widehat{}}(\xi)\right|^2 d\Gamma(\xi)\le B\|f\|_{L^2(\mu)}^2, \ \forall f\in L^2(\mu).
\end{equation}
Frame measures generalize the notion of Fourier frame in the sense that if $\{e^{2\pi i \langle\lambda,x\rangle}\}_{\lambda\in\Lambda}$ is a Fourier frame, then $\Gamma = \sum_{\lambda\in\Lambda}\delta_{\lambda}$ is the frame measure of $\mu$. The formulation of  frame measures originates from the notion of continuous frames in frame theory literature \cite{GH03}. It was studied and used in \cite{DHW14,DL14}.

\vspace{0.3cm}

%

The following lemma provides the connection between the existence of frame measure between $\mu$ and the convolution $\delta_t\ast\mu$, as well as between $\mu$, $\nu$ and the sum $\mu+\nu$. We will use it later.

\begin{lemma}\label{lem2.1}
Let $\mu$ and $\nu$ be Borel measures on ${\mathbb R}^d$. Then
\begin{enumerate}[(1)]
\item $\Gamma$ is a frame measure for $L^2(\mu)$ if and only if $\Gamma$ is a frame measure for $L^2(\delta_t\ast\mu)$ for any $t\in\mathbb{R}^d$ with the same frame bounds.
\item Suppose that $\mu(K_{\nu})=0$. If $\Gamma$ is a frame measure for $L^2(\mu+\nu)$, then $\Gamma$ is a frame measure for $L^2(\mu)$ and $L^2(\nu)$ with the same frame bounds.
\end{enumerate}
\end{lemma}

\begin{proof}
(1) Given any $f\in L^2(\delta_t\ast\mu)$,
\begin{eqnarray}\label{eq2.2}
\int|f(x)|^2 \ d\delta_t\ast\mu(x) &=& \int|f(x+t)|^2 \ d\mu(x)
\end{eqnarray}
This gives that $f (\cdot+t)\in L^2(\mu)$. On the other hand,
\begin{eqnarray*}
|(f d\delta_t\ast\mu)^{\widehat{}}(\lambda)| = | e^{-2\pi i\langle t,\lambda\rangle}(f (\cdot+t) d\mu)^{\widehat{}}(\lambda)| = |(f (\cdot+t) d\mu)^{\widehat{}}(\lambda)|\\
\end{eqnarray*}
Putting $f(\cdot+t)$ into the definition of frame measure in (\ref{eqframemeasure}), we obtain that  $\Gamma$ is a frame measure for $\delta_t\ast\mu$. The converse is similar.

\vspace{0.3cm}
%

\noindent (2) Before the proof, we notice that the statement is trivial if the support of $\mu$ and $\nu$ has an empty intersection. Our proof here does not impose the disjointness assumption. To begin, the assumption that $\Gamma$ is a frame spectrum for $L^2(\mu+\nu)$ implies there exist two constants $0 < A \le B < \infty$ such that
\begin{eqnarray}\label{e-2-1}
A\|f\|_{L^2(\mu+\nu)}^2 \le \int |(fd(\mu+\nu))^{\widehat{}}(\xi)|^2d\Gamma(\xi) \le B\|f\|_{L^2(\mu+\nu)}^2.
\end{eqnarray}
By definition, we also have  $\|f\|_{L^2(\mu)}^2+ \|f\|_{L^2(\nu)}^2 = \|f\|_{L^2(\mu+\nu)}^2$.

\vspace{0.3cm}

\noindent For any $f\in L^2(\nu)$, we may take $f(x)=0$ for any $x$ outside $K_{\nu}$. Since $\mu(K_{\nu}) = 0$, we have
 $\|f\|_{L^2(\mu)}^2 = 0$. Thus, $f\in L^2(\mu+\nu)$.  Furthermore, $(fd\mu)^{\widehat{}}(\xi) = 0$ for any $\xi\in{\mathbb R}^d$. This implies that $\|f\|_{L^2(\mu+\nu)}^2 =\|f\|^2_{L^2(\nu)}$ and  $(fd(\mu+\nu))^{\widehat{}}(\xi) = (f d\nu)^{\widehat{}}(\xi)$. By putting $f$ into (\ref{e-2-1}), we obtain
 $\Gamma$ is a frame measure for $\nu$ with frame bounds $A,B$.

\vspace{0.3cm}

\noindent For $f\in L^2(\mu)$, take $g\in L^2(\mu+\nu)$ such that $g(x) =\left\{
                                                              \begin{array}{ll}
                                                                0, & \hbox{$x\in K_{\nu}$,} \\
                                                                f(x), & \hbox{$x\in K_{\mu}\setminus K_{\nu}$.}
                                                              \end{array}
                                                            \right.$ Then
\begin{eqnarray*}
\|g\|_{L^2(\mu+\nu)}^2 &=& \|g\|_{L^2(\mu)}^2 + \|g\|_{L^2(\nu)}^2 = \|g\|_{L^2(\mu)}^2 \ \ (\text{since} \ g = 0 \ \text{on} \ K_{\nu})\\
&=& \int_{K_{\mu} - K_{\nu}} |g(x)|^2\ d\mu(x) \ (\text{since} \ \mu(K_{\nu}) = 0)\\
&=& \|f\|_{L^2(\mu)}^2. \ (\text{since} \ \mu(K_{\nu}) = 0)
\end{eqnarray*}
By a similar derivation, we also have $(g d(\mu+\nu))^{\widehat{}}(\xi) =(fd\mu)^{\widehat{}}(\xi)$.
Substituting $g$ into (\ref{e-2-1}), we obtain that $\Gamma$ is a frame measure for $L^2(\mu)$ with frame bounds $A, B$.
\end{proof}

\vspace{0.3cm}

\subsection{Packing pair.}

\vspace{0.3cm}

Let $K_1,K_2$ be two compact sets on ${\mathbb R}^d$ and let $K_1+K_2$ be the standard Minkowski sum of two sets. We say that $(K_1,K_2)$ forms  a {\it packing pair} on ${\mathbb R}^d$ if
$$
(K_1-K_1)\cap (K_2-K_2) = \{\bf 0^d\}.
$$
Given two finite Borel measures $\nu$ and $\lambda$ on ${\mathbb R}^d$. Recall that the convolution of two measures is defined as
$$
\nu\ast\lambda(E) := \int_{{\mathbb R}^d}\int_{{\mathbb R}^d}{\bf 1}_E(x+y)d\nu(x)d\lambda(y),
$$
where ${\bf 1}_E$ denotes the characteristic function on a set $E$. We will also assume that $\nu$ and $\lambda$ are continuous measures without atoms. Denote by $K_{\nu}$ and $K_{\lambda}$ the support of $\nu$ and $\lambda$ respectively. Then $\mu: = \nu\ast\lambda$ has support
$$
K_{\mu} = K_{\nu}+K_{\lambda}.
$$
We say that the pair $(\nu,\lambda)$ is a {\it packing pair}  if the support $(K_\nu,K_{\lambda})$ forms a packing pair.

\begin{lemma}\label{lemma_packing}
Let $(K_1,K_2)$ be a packing pair. Then
\begin{enumerate}
\item For any $x,y\in K_{2}$ and $x\ne y$, $(K_{1}+x)\cap (K_{1}+y) = \emptyset$.
\item For any $E\subset K_1$ and $F\subset K_2$, ${\bf 1}_{E+F} (x+y) = {\bf 1}_{E} (x){\bf 1}_F(y)$ for any $x\in K_1$ and $y\in K_2$.
\item Suppose that $(\nu,\lambda)$ is a packing pair and $\mu = \nu\ast\lambda$. Then for any $E\subset K_{\nu}$ and $F\subset K_{\lambda}$,
\begin{equation}\label{eq2.4}
{({\bf 1}_{E+F}d\mu)^{\widehat{}}}(\xi) = (({\bf 1}_Ed\nu)\ast({\bf 1}_Fd\lambda))^{\widehat{}}(\xi) = ({\bf 1}_Ed\nu)^{\widehat{}}(\xi)({\bf 1}_Fd\lambda)^{\widehat{}}(\xi) .
\end{equation}
In particular,
\begin{equation}\label{eq2.5}
\mu(E+F) = \nu(E)\lambda(F).
\end{equation}
\end{enumerate}
\end{lemma}

\begin{proof}
For (1), if the intersection is non-empty. we can find $k_1,k_2\in K_{1}$ such that $k_1+x=k_2+y$. Then $x-y = k_2-k_1\in  (K_{2}-K_{2})\cap (K_{1}-K_{1})$. By the assumption of packing pair, $x=y$.

\vspace{0.3cm}

For (2), if $x+y\not\in E+F$, then $x\not\in E$ or $y\not\in F$. In this case, ${\bf 1}_{E+F} (x+y) =0 = {\bf 1}_E(x){\bf 1}_F (y)$. On the other hand, if $x+y\in E+F$, then $x+y = e+f$ for some $e\in E\subset K_1$ and $f\in F\subset K_2$. This implies that $x-e = f-y\in (K_1-K_1)\cap(K_2-K_2)$.  By the assumption of packing pair, $x=e\in E$ and $y=f\in F$. This shows that ${\bf 1}_{E+F} (x+y) =1 = {\bf 1}_E(x){\bf 1}_F (y)$.

\vspace{0.3cm}

For (3), the second equality of (\ref{eq2.4}) is just the standard property of Fourier transform of measures. For the first equality, we note that
$$
\begin{aligned}
(({\bf 1}_Ed\nu)\ast({\bf 1}_Fd\lambda))^{\widehat{}}(\xi) = &\int\int {\bf 1}_E(x){\bf 1}_F(y)e^{-2\pi i \langle \xi, x+y\rangle}d\nu(x)d\lambda(y)\\
=&\int\int {\bf 1}_{E+F}(x+y)e^{-2\pi i \langle \xi, x+y\rangle}d\nu(x)d\lambda(y)\\
=&  ({\bf 1}_{E+F}d\mu)^{\widehat{}}(\xi).
\end{aligned}
$$
The second equality above uses (2) since $x\in K_{\nu}$ and $y\in K_{\lambda}$ with $(K_{\nu},K_{\lambda})$ forms a packing pair. Finally, (\ref{eq2.5}) follows from (\ref{eq2.4}) by taking $\xi = 0$.
\end{proof}

\vspace{0.3cm}
\begin{lemma}\label{lem2.2}
Let  $\nu$ and $\lambda$ be finite Borel measures without atoms (i.e. all points have measure zero) forming a packing pair and  $\mu = \nu\ast\lambda$. Then, for any $t\in {\mathbb R}^d$, $\mu(K_{\nu}+t) = 0$.
\end{lemma}

\begin{proof}
If $t\in K_{\lambda}$, then it follows directly from (\ref{eq2.5}) that $\mu(K_{\nu}+x) = \nu(K_{\nu})\lambda(\{x\}) = 0$ since $\lambda$  has no atoms. Suppose that there exists $t\in{\mathbb R}^d$ such that $\mu(K_{\nu}+t)>0$. By definition,
$$
0< \mu (K_{\nu}+t) = \int \nu (K_{\nu}+t-x)d\lambda (x).
$$
Hence, the set $E = \{x\in K_{\lambda}:\nu (K_{\nu}+t-x) >0 \}$ will have a positive $\lambda$-measure (i.e. $\lambda(E)>0$). As $\lambda$ has no atom, $E$ is an uncountable set. In particular, there exists $n_0$ such that the set
$$
E_{n_0}: = \left\{x\in E: \nu (K_{\nu}+t-x) \ge \frac{1}{n_0} \right\}
$$
is uncountable. Take any countable sequence of distinct points $x_1,...,x_n,...$ from $E_{n_0}$. As $(K_{\nu},K_{\lambda})$ forms a packing pair, so is $(K_{\nu},-K_{\lambda})$. By Lemma \ref{lemma_packing} (1), $(K_{\nu}-x_i)\cap (K_{\nu}-x_j) = \emptyset$ for $i\ne j$. Thus, $(K_{\nu}+t-x_i)\cap (K_{\nu}+t-x_j) = \emptyset$.   Hence, by the monotonicity and disjoint countable additivity of measure,
$$
\begin{aligned}
\nu({\mathbb R}^d)\ge \nu \left(\bigcup_{n=1}^{\infty}\left((K_{\nu}+t-x_n)\right)\right)  \ge&  \nu \left(\bigcup_{n=1}^N\left((K_{\nu}+t-x_n)\right)\right)\\
=&\sum_{n=1}^N\nu\left((K_{\nu}+t-x_n)\right) \ge \frac{N}{n_0}.
\end{aligned}
$$
Since $N$ is arbitrary, $\nu ({\mathbb R}^d) = \infty$. This is a contradiction since $\nu$ is a finite Borel measure. This  establishes the lemma.
\end{proof}

\vspace{0.3cm}

Suppose that $(\nu,\lambda)$ forms a packing pair and $\mu = \nu\ast\lambda$. We also assume that $0\in K_{\lambda}$. Otherwise, we may translate our measures and consider $\delta_{-x}\ast \mu = \nu\ast (\delta_{-x}\ast\lambda)$ with $x\in K_{\lambda}$. In this translation, $0\in K_{\lambda}$. In other words, for any open ball ${\mathcal B}$ centered around $0$, we have $\lambda({\mathcal B})>0$ if $\lambda$ has no atoms.

\vspace{0.3cm}

For $t\in{\mathbb R}^d$, we define a sum of measure by
$$
\rho_t  = \mu + \delta_t\ast\nu.
$$
Now, $K_{\rho_t}$ =  $K_{\mu}\cup(K_{\nu}+t)$. As $0\in K_{\lambda}$, we have $K_{\nu}\subset K_{\mu}$. The following lemma will be needed in the proof of our main theorem.

\begin{lemma}\label{lem2.3}
Let $\nu$ and $\lambda$ be two finite Borel measures without atoms and $(\nu,\lambda)$ forms a packing pair. Let also $\mu = \nu\ast\lambda$. Then for any $t\in{\mathbb R}^d$, there exists $x\in K_{\lambda}$ such that
$$
\nu ((K_{\nu}+t)\cap (K_{\nu}+x)) =0.
$$
\end{lemma}

\begin{proof}
Let $t\in{\mathbb R}^d$ be fixed.  Then for any $x\in K_{\lambda}$, $\nu ((K_{\nu}+t)\cap (K_{\nu}+x))>0$. As $\lambda$ is a measure without atoms, $K_{\lambda}$ has uncountably many points. In particular, there exists a positive integer $n_0$ such that the set
$$
E_{n_0}: = \left\{x\in K_{\lambda} : \nu ((K_{\nu}+t)\cap (K_{\nu}+x)) \ge \frac{1}{n_0}\right\}
$$
has uncountably many points. Take any countable sequence of distinct points $x_1,...,x_n,...$ from $E_{n_0}$. Now,
$$
K_{\mu} = K_{\nu}+K_{\lambda} = \bigcup_{x\in K_{\lambda}} (K_{\nu}+x) \supset \bigcup_{n=1}^N(K_{\nu}+x_n),
$$
for all integers $N$. Moreover, as $(\nu,\lambda)$ forms a packing pair, $(K_{\nu}+x_i)\cap (K_{\nu}+x_j) = \emptyset$ for $i\ne j$ by Lemma \ref{lemma_packing} (1).  Hence, by the monotonicity and disjoint countable additivity of measure,
$$
\begin{aligned}
\nu (K_{\mu}\cap (K_{\nu}+t))  \ge&  \nu \left(\bigcup_{n=1}^N\left((K_{\nu}+x_n)\cap (K_{\nu}+t)\right)\right)\\
=&\sum_{n=1}^N\nu\left((K_{\nu}+x_n)\cap (K_{\nu}+t)\right) \ge \frac{N}{n_0}.
\end{aligned}
$$
Since $N$ is arbitrary, $\nu ({\mathbb R}^d) \ge \nu (K_{\mu}\cap (K_{\nu}+t)) = \infty$. This is a contradiction since $\nu$ is a finite Borel measure. This establishes the lemma.
\end{proof}

\vspace{0.3cm}
\begin{lemma}\label{lemBorel}
Let $\rho = \mu+\nu$ and let $X$ be a Borel support of $\rho$. Suppose that $\mu(K_{\nu}) = 0$. Then $X\cap K_{\mu}$ is a Borel support of $\mu$ and  $X\cap K_{\nu}$ is a Borel support of $\nu$.
\end{lemma}

\begin{proof}
We first note that if $X$ is a Borel support of a measure $\mu$. Then $\mu(E) = \mu(X\cap E) $ for any Borel set $E$. Now,
$$
\nu(K_{\nu}) = \rho(K_{\nu}) = \rho(K_{\nu}\cap X) = \mu(K_{\nu}\cap X)+\nu(K_{\nu}\cap X) = \nu(K_{\nu}\cap X).
$$
Thus, $K_{\nu}\cap X$ is a Borel support of $\nu$. For the measure $\mu$.
$$
\begin{aligned}
\mu(K_{\mu})  =& \mu(K_{\mu}\setminus K_{\nu})= \rho(K_{\mu}\setminus K_{\nu}) = \rho((K_{\mu}\setminus K_{\nu})\cap X)  \\ =&\mu((K_{\mu}\setminus K_{\nu})\cap X)+\nu((K_{\mu}\setminus K_{\nu})\cap X) = \mu(K_{\mu}\cap X).
\end{aligned}
$$
\end{proof}
\vspace{0.3cm}

We are now ready to prove Theorem \ref{t-1-1} which is restated as follows.

\begin{theorem}\label{th2.1}
Let $(\nu,\lambda)$ be a packing pair and $\nu,\lambda$ have no atoms with $\mu = \nu\ast\lambda$. Then $\rho_t$ is translationally singular and it does not admit any frame measure.
\end{theorem}

\begin{proof}
Let $E = K_{\mu}$ which is inside $K_{\rho_t}$. Then $\rho_t(E) \ge \mu(E) = 1 > 0$. We will prove that $\rho_t$ is translationally singular on $E$.  Using Lemma \ref{lem2.3}, we can take an  $x\in K_{\lambda}$ such that $\nu((K_{\nu}+x)\cap (K_{\nu}+t)) = 0$.  Consider $F = K_{\nu}+x \subset E$ (as $x\in K_{\lambda}$). Since the support of $\delta_t\ast \nu$ is $K_{\nu}+t$, we now have $\delta_t\ast\nu(K_{\nu}+x) = \nu((K_{\nu}+x)\cap (K_{\nu}+t)) = 0$. Meanwhile, $\mu(K_{\nu}+x) = 0$ by Lemma \ref{lem2.2}. Thus,
$$
\rho_t(K_{\nu}+x) = \mu(K_{\nu}+x)+ \delta_t\ast\nu(K_{\nu}+x) = 0.
$$
On the other hand,
\begin{eqnarray*}
\omega_{t-x}(K_{\nu}+x) &=& \rho_t((K_{\nu}+x+(t-x))\cap (K_{\mu}+(t-x)))\\
&= & \mu ((K_{\nu}+t)\cap (K_{\mu}+(t-x)))+ \nu (K_{\nu}\cap (K_{\mu}-x))\\
&\ge &\nu(K_{\nu}\cap (K_{\mu}-x)).
\end{eqnarray*}
Note that  $K_{\nu}+x\subset K_{\mu}$, so we have $(K_{\nu}+x)\cap K_{\mu} = K_{\nu}+x$. Hence, after a translation of $x$, $K_{\nu}\cap (K_{\mu}-x) = K_{\nu}$ and
$$
\nu(K_{\nu}\cap (K_{\mu}-x)) = \nu (K_{\nu})=1.
$$
Thus $\omega_{t-x}(K_{\nu}+x)>0$. This shows that $\omega_{t-x}$ is singular with respect to $\mu$.  $\rho_t$ is not  translationally absolutely continuous on $E = K_{\mu}$. Thus,  $\rho_t$ is  translationally singular if we consider the support to be $K_{\rho_t}$.

\vspace{0.3cm}

In general, let $X$ be a Borel support of $\rho_t$. Then Lemma \ref{lemBorel} implies that  $X = X_{\mu}\cup X_{\nu}$, where $X_{\mu}$ is a Borel support of $\mu$ and $X_{\nu}$ is a Borel support of $\nu$. We replace all $K_{\mu}$ by  $X_{\mu}$ and $K_{\nu}$ by  $X_{\nu}$ and repeat the proof in the previous paragraph. Since $X_{\mu}\subset K_{\mu}$ and $X_{\nu}\subset{K_{\nu}}$, the proof remains the same. Thus  $\rho_t$ is translationally singular.

\vspace{0.3cm}


\vspace{0.3cm}

We now show that $\rho_t$ does not admit any frame measure. Suppose on the contrary that $\rho_t$ admits a frame measure $\Gamma$ with frame bounds $0< A \le B < \infty$. As discussed before Lemma \ref{lem2.3}, we may assume $0\in K_{\lambda}$. Let
$$
V_n = K_{\nu} +({\mathcal B}(0,1/n)\cap K_{\lambda})
,$$
where ${\mathcal B}(0,1/n)$ denotes the ball of radius $1/n$ centered at $0$. Note that $0\in K_{\lambda}$. This means that $\lambda({\mathcal B}(0,1/n))>0$ for any $n$. Now, $V_n\subset K_{\mu}$ and  ${\bf 1}_{V_n}\in L^2(\mu)$. By Lemma \ref{lem2.1}, with $\mu(K_{\nu}+t)=0$ in Lemma \ref{lem2.2}, $\Gamma$ is also a frame measure for $L^2(\mu)$ and $L^2(\delta_t\ast\nu)$ for any $t\in{\mathbb R}^d$ with the same bounds $A\le B$. Therefore, we have
\begin{eqnarray}\label{eq2.7}
A\cdot\mu(V_n) = A\int |{\bf 1}_{V_n}(x)|^2 \ d\mu(x)
&\le& \int |({\bf 1}_{V_n}d\mu)^{\widehat{}}(\xi) |^2 \ d\Gamma(\xi) \nonumber\\
&=& \int |({\bf 1}_{K_{\nu}} \ d\nu)^{\widehat{}}(\xi) |^2  |({\bf 1}_{{\mathcal B}(0,1/n)} \ d\lambda)^{\widehat{}}(\xi) |^2 \ d\Gamma(\xi)\nonumber\\
&\le &  (\lambda({\mathcal B}(0,1/n)))^2 \int |({\bf 1}_{K_{\nu}} \ d\nu)^{\widehat{}}(\xi) |^2 \ d\Gamma(\xi)\\
&\le& B (\lambda({\mathcal B}(0,1/n)))^2 \int |({\bf 1}_{K_{\nu}}|^2 \ d\nu\nonumber\\
& = & B\cdot(\lambda({\mathcal B}(0,1/n)))^2\nonumber.
\end{eqnarray}
By (\ref{eq2.5}),  $\mu(V_n) = \nu(K_{\nu})\lambda({\mathcal B}(0,1/n) \cap K_{\lambda}) = \lambda({\mathcal B}(0,1/n))$. (\ref{eq2.7}) implies that
$$A\lambda({\mathcal B}(0,1/n)) \le B(\lambda({\mathcal B}(0,1/n)))^2,$$
 which is equivalent to
\begin{eqnarray}\label{eq2.8}
 \frac{1}{\lambda({\mathcal B}(0,1/n))} \le \frac{B}{A}.
\end{eqnarray}
As $\lambda$ has no atoms, $\lambda({\mathcal B}(0,1/n))$ approaches zero as $n$ tends to infinity. This means that the left hand side of (\ref{eq2.8}) can be made arbitrarily large, while the right hand side remains bounded. This is a contradiction. Hence, we conclude that $\rho_t$ does not admit any frame measure.
\end{proof}

\section{Examples of packing pair}
In this section, we show that packing pairs exist in many circumstances using self-affine sets and self-affine measures.

\vspace{0.3cm}

Let $R\in M_d(\mathbb{Z})$ be a $d\times d$ \emph{expanding matrix} (all its eigenvalues $\lambda_i$ satisfy
$\lvert\lambda_i\rvert>1$) with integer entries and let $B\subset\mathbb{Z}^d$ be a finite set of
integer vectors, which we call \emph{a digit set}.
Given a pair $(R,B)$, define
$$
\tau_b(x)=R^{-1}(x+b), \ b\in B.
$$
The family of maps $\{\tau_b(x)\}_{b\in B}$ is called an \emph{iterated function system} (IFS). The invariant attractor $T(R,B)$ determined by $(R, B)$ is the unique non-empty compact set satisfying
$$
K_{R,B}=\bigcup\limits_{b\in B}\tau_b(K_{R,B}),
$$
or equivalently,
\begin{eqnarray}\label{e-2-6}
K_{R,B} = \left\{\sum_{n=1}^{\infty}R^{-n}b_n: b_n\in B\right\}.
\end{eqnarray}
The set $K_{R,B}$ is called a {\it self-affine set}. If $R = \rho O$, where $\rho>1$ and $O$ is an orthogonal matrix, then $K_{R,B}$ is called  a self-similar set.  We say that the IFS $\{\tau_b\}_{b\in B}$ satisfies the \emph{strong separation condition} (SSC) if $\tau_b(K_{R,B})\cap \tau_{b'}(K_{R,B})$ are all disjoint for all $b\ne b'$ (See \cite{F} for details).

\vspace{0.3cm}

The self-affine measure determined by the pair $(R, B)$ (with equal weights) is the unique probability measure $\mu:=\mu_{R,B}$ satisfying
$$
\mu(E) = \sum\limits_{b\in B} \frac{1}{\# B} \mu(\tau_b^{-1}(E)) \ \text{for any Borel set} \ E\subset \mathbb{R}^d.
$$
 The measure can also be realized as an infinite convolutions of Dirac measures:
\begin{eqnarray}\label{e-2-5}
\mu_{R,B} = \delta_{R^{-1}B}\ast\delta_{R^{-2}B}\ast\delta_{R^{-3}B}\ast\dotsm,
\end{eqnarray}
where $\delta_{A} = \frac{1}{\#A}\sum_{a\in{A}}\delta_a$.

\vspace{0.3cm}

\subsection{Two criteria for packing pair.} We will give two simple criteria  for packing self-affine pair. The following proposition offers the first one.
\begin{proposition}\label{prop3.1}
Let $K_{R,B}$ and $K_{R,C}$ be two compact invariant attractors generated by the expanding matrix $R\in M_{d}({\mathbb Z})$ and the digit sets $B, C\subset {\mathbb Z}^d$ respectively. Suppose that the digit sets $B,C$ satisfy
\begin{enumerate}
\item $(B-B)\cap (C-C) = \{0\}$ (i.e. $(B, C)$ forms a packing pair),
\item Let $D: = \max\{|d|: d\in (B-B) - (C-C)\}$ and $\frac{D\|R^{-1}\|}{1-\|R^{-1}\|}<1$,  where $|\cdot|$ denotes the standard Euclidean norm and $\|R\|$ denotes the norm of the matrix $R$.
\end{enumerate}
Then $(K_{R,B},K_{R,C})$ forms a packing pair.
\end{proposition}

\begin{proof}
We need to show that $(K_{R,B}-K_{R,B})\cap (K_{R,C}-K_{R,C}) = \{0\}$. To prove this, let $x\in(K_{R,B}-K_{R,B})\cap (K_{R,C}-K_{R,C})$. Then using the representation in  (\ref{e-2-6})
$$
x = \sum_{n=1}^{\infty} R^{-n}\epsilon_n = \sum_{n=1}^{\infty} R^{-n}\eta_n,
$$
where $\epsilon_n\in B-B$ and $\eta_n\in C-C$. Hence,
$$
\sum_{n=1}^{\infty} R^{-n}(\epsilon_n-\eta_n) = 0.
$$
We now multiply both side by $R$ and obtain
$$
\epsilon_1 - \eta_1 = \sum_{n=1}^{\infty} R^{-{n}}(\epsilon_{n+1}-\eta_{n+1}).
 $$
Thus,
$$
|\epsilon_1 - \eta_1|\le \sum_{n=1}^{\infty}\|R^{-1}\|^n D = \frac{D\|R^{-1}\|}{1-\|R^{-1}\|}<1.
$$
But $\epsilon_1,\eta_1\in {\mathbb Z}$, we must have $\epsilon_1=\eta_1$. Note that $\epsilon_1 = \eta_1\in (B-B)\cap (C-C)$. By assumption, $\epsilon_1 = \eta_1 = 0$. Thus, $x = \sum_{n=2}^{\infty} R^{-n}\epsilon_n = \sum_{n=2}^{\infty} R^{-n}\eta_n$. We repeat the procedure and we finally obtain $\epsilon_n = \eta_n = 0$ for all $n$. This shows that $x = 0$. Thus, $(K_{R,B},K_{R,C})$ forms a packing pair.

\end{proof}

\begin{example}\label{example3.2}
{ Suppose that $R = N$ with $N\in{\mathbb Z}$ and $N\ge 11$,
$B = \{0,1\}$ and $C = \{0,4\}$. Then $(K_{N,B},K_{N,C})$ forms a packing pair.}
\end{example}

\begin{proof}
With the notations in Proposition \ref{prop3.1}, note that
$$
(B-B)\cap (C-C) = \{-1,0,1\}\cap\{-4,0,4\} = \{0\}.
$$
 Also, $(B-B)-(C-C) = \{-5,-4,-3,-1,0,1,3,4,5\}$ and thus $D = 5$. We have $ \frac{D\|R^{-1}\|}{1-\|R^{-1}\|} = \frac{D/N}{1-D/N} = \frac{5}{N-5}<1$ if $N\ge 11$.

\end{proof}

\vspace{0.3cm}

We now give another useful criterion for packing pair. Let $S$ be an infinite subset of positive integers such that ${\mathbb Z}^+\setminus S$ is also an infinite set, where ${\mathbb Z}^+$ is the set of all positive integers. For a self-affine set $K_{R,B}$ given in (\ref{e-2-6}), we can define
$$
K_S = \left\{\sum_{n\in S}R^{-n}b_n: b_n\in B\right\}, \ K_{{\mathbb Z}^+\setminus S} = \left\{\sum_{n\in {\mathbb Z}^+\setminus S}R^{-n}b_n:b_n\in B\right\}.
$$

\begin{proposition}\label{prop3.2}
Suppose that the IFS generated by $(R,B)$ satisfies the strong separation condition. Then $(K_S,K_{{\mathbb Z}^+\setminus S})$ forms a packing pair.
\end{proposition}

\begin{proof}
Take $x\in (K_S-K_S)\cap (K_{{\mathbb Z}^+\setminus S}-K_{{\mathbb Z}^+\setminus S})$. Then
$$
x = \sum_{n\in S}R^{-n}(b_n-b_n') = \sum_{n\in {\mathbb Z}^+\setminus S}R^{-n}(b_n'-b_n).
$$
Hence, by a rearrangement, we have
$$
\sum_{n=1}^{\infty}R^{-n}b_n = \sum_{n=1}^{\infty}R^{-n}b_n' \in \tau_{b_1}(K_{R,B})\cap \tau_{b_1'}(K_{R,B}).
$$
We note that if the IFS generated by $(R, B)$ satisfies the strong separation condition, then $\tau_{b_1}(K_{R,B})\cap \tau_{b_1'}(K_{R,B})=\emptyset$ if $b_1\ne b_1'$. Hence, $b_1 = b_1'$. We repeat the proof and we obtain $b_n = b_n'$ for all $n$. Therefore, $x = 0$. This shows $(K_S,K_{{\mathbb Z}^+\setminus S})$ forms a packing pair.
\end{proof}

Combining the above propositions and Theorem \ref{th2.1}. We obtain the following theorem.

\begin{theorem}\label{theorem3.1}
Let $\nu$ and $\lambda$ are two finite Borel measures and define $\mu = \nu\ast\lambda$. Suppose that one of the following is satisfied:
\begin{enumerate}
\item if  $K_{\nu} = K_{R,B}$ and $K_{\lambda} = K_{R,C}$, where $R,B,C$ satisfies the assumption in Proposition \ref{prop3.1},  or
\item if  $K_{\nu} = K_{S}$ and $K_{\lambda} = K_{{\mathbb Z}^+\setminus S}$, where
 the IFS generated by $(R,B)$ satisfies the strong separation condition and $S$ is an infinite subset of integers with ${\mathbb Z}^+\setminus S$ is also an infinite set.
\end{enumerate}
Then for any $t\in{\mathbb R}^d$,  the measure $\mu + \delta_t\ast \nu$
is translationally singular and it does not admit any frame measures or Fourier frames.
\end{theorem}

\vspace{0.3cm}
\subsection{A closer look to $\mu_4$ and $\mu_{16}$.} Self-affine measures with proposition \ref{prop3.1} can be found easily. In particular, we are interested in the $\mu_4$ and $\mu_{16}$ example given in the introduction. As $\mu_4$ is a self-similar measure, we can write it as
$$
\begin{aligned}
\mu_4 =& \delta_{4^{-1}\{0,1\}}\ast\delta_{4^{-2}\{0,1\}}\ast....\\
=& (\delta_{4^{-2}\{0,1\}}\ast \delta_{4^{-4}\{0,1\}}....) \ast (\delta_{4^{-1}\{0,1\}}\ast \delta_{4^{-3}\{0,1\}}....)\\
=& \mu_{16}\ast  (\delta_{4^{-2}\{0,4\}}\ast \delta_{4^{-4}\{0,4\}}....)\\
=& \mu_{16}\ast \lambda_{16}
\end{aligned}
$$
where $\lambda_{16}$ is the self-similar measure generated by the maps $\frac{1}{16}x$ and $\frac{1}{16}(x+4)$.

\vspace{0.3cm}

\begin{proposition}\label{prop3.5}
The support of $\mu_{16}$ and $\lambda_{16}$ forms a packing pair and the measure $\mu_4+\delta_x\ast\mu_{16}$ is translationally singular and does not admit any frame measures.
\end{proposition}
\begin{proof}
It is known that the supports of $\mu_{16}$ and $\lambda_{16}$, denoted respectively by $K_{\mu_{16}}$ and $K_{\lambda_{16}}$, are generated by the pairs $(16,\{0, 1\})$ and  $(16,\{0, 4\})$ respectively. It follows from Example \ref{example3.2} that $(K_{\mu_{16}},K_{\lambda_{16}})$ forms a packing pair and thus $(\mu_{16},\lambda_{16})$ forms a packing pair. Using Theorem \ref{theorem3.1} and note that $\mu_4 =\mu_{16}\ast \lambda_{16}$, our second statement follows.
\end{proof}

\section{Mixed type measures with a Fourier frame}
In this section, we generalize the construction of Lev in Theorem \ref{t-1-3} to non-orthogonal subspace and prove Theorem \ref{t-1-4}.

\vspace{0.3cm}

Let $\mu$ and $\nu$ be positive and finite continuous Borel measures supported on ${\mathbb R}^m$ and ${\mathbb R}^{d-m}$ respectively. Suppose that they are frame spectral measures on their respective spaces. Let $V$ be a subspace of dimension $d$ and let $T: {\mathbb R}^d \rightarrow {\mathbb R}^d$ be the invertible linear transformation such that
\begin{eqnarray*}
T(\{{\bf 0}^{m}\} \times {\mathbb R}^{d-m}) = V.
\end{eqnarray*}
Define $\nu_T(E) = (\delta_{{\bf 0}^m}\times\nu)(T^{-1}(E))$. Our goal is to show that $\mu\times \delta_{{\bf 0}^{d-m}}+\nu_T$ is frame-spectral if $V\cap ({\mathbb R}^m\times \{{\bf 0}^{d-m}\}) = \{{\bf 0}^d\}$. We first need the following two lemmas from linear algebra to characterize the condition $V\cap ({\mathbb R}^m\times \{{\bf 0}^{d-m}\}) = \{{\bf 0}^d\}$.

\vspace{0.3cm}

\begin{lemma}\label{t-3-1}
Let ${\bf e}_j = (0, \dotsc, 1,  \dotsc, 0)\in {\mathbb R}^d$ with the $j$-th entry is $1$ and the others are $0$. $V\cap ({\mathbb R}^m\times \{{\bf 0}^{d-m}\}) = \{{\bf 0}^{d}\}$ if and only if $\{{\bf e}_1, {\bf e}_2, \dotsc, {\bf e}_m, T({\bf e}_{m+1}), \dotsc, T({\bf e}_d)\}$ is linearly independent.
\end{lemma}
\begin{proof} Suppose that $V\cap ({\mathbb R}^m\times \{{\bf 0}^{d-m}\}) = \{{\bf 0}^d\}$.  Note that $T({\bf e}_j) \in V$ for all $j = m+1, \dotsc, d$. Thus if
$\sum\limits_{i=1}^m \alpha_i {\bf e}_i + \sum\limits_{j= m+1}^d \beta_j (T{\bf e}_j) = {{\bf 0}^d}$, then
\begin{eqnarray*}
\sum\limits_{i=1}^m \alpha_i {\bf e}_i = -\sum\limits_{j=m+1}^d \beta_j (T{\bf e}_j) \in V\cap ({\mathbb R}^m\times\{{\bf 0}^{d-m}\}).
\end{eqnarray*}
Therefore,
$$\sum\limits_{i=1}^m \alpha_i {\bf e}_i = {\bf 0}^d = - \sum\limits_{j=m+1}^d \beta_j (T{\bf e}_j).$$
 As $\{{\bf e}_1, \dotsc, {\bf e}_m\}$ is linearly independent and $\{T({\bf e}_{m+1}), \dotsc, T({\bf e}_d)\}$ is also linearly independent (since $T$ is invertible), we have $\alpha_i = \beta_j = 0$ for any $1\le i\le m$ and $m+1\le j\le d$.

\vspace{0.3cm}

Conversely, we take ${\bf x} \in V\cap ({\mathbb R}^m\times\{{\bf 0}^{d-m}\})$. Then
${\bf x} = \sum\limits_{i=1}^m \alpha_i {\bf e}_i = \sum\limits_{j=m+1}^d \beta_j (T{\bf e}_j)$. This implies that
\begin{eqnarray*}
\sum\limits_{i=1}^m \alpha_i {\bf e}_i - \sum\limits_{j=m+1}^d \beta_j (T{\bf e}_j) = {\bf 0}^d.
 \end{eqnarray*}
 Thus $\alpha_i = \beta_j = 0 $ for any $1\le i\le m$ and $m+1\le j\le d$ by the assumption. Hence, ${\bf x} = {\bf 0}^d.$
\end{proof}

Now we represent $T$ in its matrix representation:
$$
T = \begin{bmatrix} T{\bf e}_1 \ T{\bf e}_2 \ \dotsm \ T{\bf e}_d\end{bmatrix}  = \quad \left[
                                                                                \begin{array}{ccc}
                                                                                  A_1 & | & A_2 \\
                                                                                  -- & | & -- \\
                                                                                  A_3 & | & A_4 \\
                                                                                \end{array}
                                                                              \right]
$$
 where $A_1$ is an $m\times m$ matrix and $A_4$ is a $(d-m)\times (d-m)$ matrix.

\vspace{0.3cm}

\begin{lemma}\label{t-3-2}
$V\cap ({\mathbb R}^m\times \{{\bf 0}^{d-m}\}) = \{{\bf 0}^d\}$ if and only if $A_4$ is invertible.
\end{lemma}
\begin{proof}
From Lemma \ref{t-3-1},  $V\cap ({\mathbb R}^m\times \{{\bf 0}^{d-m}\}) = \{{\bf 0}^d\}$ happens if and only if
$$
\{{\bf e}_1, {\bf e}_2, \dotsc, {\bf e}_m, T({\bf e}_{m+1}), \dotsc, T({\bf e}_d)\}
$$
is linearly independent. Hence, this is equivalent to
\begin{eqnarray*}
0&\ne& \det \begin{bmatrix}{\bf e}_1 \ {\bf e}_2 \ \dotsc \ {\bf e}_m \ T({\bf e}_{m+1}) \ \dotsc \ T({\bf e}_d)\end{bmatrix}\quad\\
&=& \det \left[
                                                                                \begin{array}{ccc}
                                                                                  I & | & A_2 \\
                                                                                  -- & | & -- \\
                                                                                  O & | & A_4 \\
                                                                                \end{array}
                                                                              \right] = \det A_4,
\end{eqnarray*}
which is equivalent to $A_4$ is invertible. The lemma follows.
\end{proof}

\vspace{0.3cm}

We can now prove our Theorem \ref{t-1-4}, which is restated as follows:
\vspace{0.3cm}

\begin{theorem}\label{t-3-3}
Suppose that $V$ is a subspace of ${\mathbb R}^d$ such that $V\cap ({\mathbb R}^m\times \{{\bf 0}^{d-m}\}) = \{{\bf 0}^d\}$ and the linear transformation $T$ which maps $(\{{\bf 0}^{m}\} \times {\mathbb R}^{d-m})$ onto $ V$ is invertible. Then
$$
\rho_T: = \mu \times \delta_{{\bf 0}^m} + \nu_T
$$
 is a frame-spectral measure. Furthermore, there exists $0<A\le B<\infty$ such that  there exists a frame spectrum $\Lambda_T$ for $\rho_T$ with frame bounds $A,B$.
\end{theorem}

\begin{proof}
For $f\in L^2(\rho_T)$, let $g(x,y) = f(x+A_2A_4^{-1}y,y)$ which is well-defined since $A_4$ is invertible from Lemma \ref{t-3-2}.  $f$ can be decomposed as
\begin{eqnarray*}
f(x,y) d\rho_T(x,y) &=& f(x,0) \ d \mu(x) + f(T\begin{pmatrix}
0 \\ y \end{pmatrix}) \ d\nu(y)\\
&=& f(x,0) \ d \mu(x) + f(A_2y,A_4y) \ d\nu(y)\\
&=& f(x,0) \ d \mu(x) + f(A_2A_4^{-1}y,y) \ d(\nu A_4^{-1})(y)\\
&=& g(x,0) \ d \mu(x) + g(0,y) \ d(\nu A_4^{-1})(y) = g(x,y)d\rho(x,y)
\end{eqnarray*}
where $\nu A_4^{-1}(E): = \nu(A_4^{-1}(E))$ for any Borel sets $E$ and
$$
 \rho:=\mu \times \delta_{{\bf 0}^m} + \delta_{{\bf 0}^{d-m}}\times(\nu A_4^{-1}).
$$
In view of the above, $f\in L^2(\rho_T)$ if and only if $g\in L^2(\rho)$. As $A_4$ is invertible, by a change of variable,  $\Lambda_{\nu}$ is a frame spectrum for $\nu$ with frame bound $A,B$ if and only if $(A_4^{t})^{-1}\Lambda_{\nu}$ is a frame spectrum for $\nu A_4^{-1}$  with frame bound $A,B$ . Using Theorem \ref{t-1-3}, the measure $\rho$ is a  frame-spectral measure. Therefore, $\rho_T$ is a frame spectral measure.   Moreover, according to Lev's construction, the frame bound for the frame spectrum of $\rho$ depends only on the frame bounds of $\mu$ and $\nu A_4^{-1}$. But the frame bounds of the frame spectrum $(A_4)^{-1}\Lambda_{\nu}$ for $\nu A_4^{-1}$ is independent of $A_4$. We have that the frame spectrum we constructed for $\rho_T$ is independent of $T$. This completes the whole proof.
 \end{proof}

Recall that a frame-spectral measure must be either purely discrete, purely absolutely continuous or purely singularly continuous with respect to Lebesgue measure \cite{HLL}.  We note that the assumption that $\mu$ and $\nu$ are continuous measures and $V\cap ({\mathbb R}^m\times \{{\bf 0}^{d-m}\}) = \{{\bf 0}^d\}$ cannot be removed from Theorem \ref{t-3-3}, it can be proved using the pure-type theorem stated.

\begin{example}\label{example4}
We denote by ${\mathcal L}^{s}_A$ the $s$-dimensional Lebesgue measure supported on $A$.
\begin{enumerate}
\item {\rm Suppose that $\mu = \delta_{(0,1)}$ and $\nu  = {\mathcal L}^1_{[1,2]\times\{0\}}$. Then $\mu+\nu$ does not admit any Fourier frames since $\mu+\nu$ is not of pure type (\cite{HLL}). This shows that the condition $\mu$ and $\nu$ are continuous measures cannot be removed}

\vspace{0.3cm}

\item {\rm For the condition $V\cap ({\mathbb R}^m\times \{{\bf 0}^{d-m}\}) = \{{\bf 0}^d\}$, one may think $\mu_4+\mu_{16}$ is a counterexample. On the other hand, if we consider $\mu = \mu_1+\mu_2$, where $\mu_1 = {\mathcal L}^1_{\{0\}\times[0,1]\times\{0\}}$ and $\mu_2 = {\mathcal L}^1_{\{0\}\times\{0\}\times[0,1]}$, and $\nu={\mathcal L}^2_{[0,1]^2\times \{0\}}$. Then $\mu$ is supported on the $yz$-plane and $\nu$ is supported on the $xy$-plane. Furthermore, $\mu$,$\nu$ admit a Fourier frame ($\mu$ admits a Fourier frame by Theorem \ref{t-3-3}). }

\medskip

     {\rm However, $\mu+\nu$ does not admit any Fourier frame on ${\mathbb R}^3$. To see this, any Fourier frame $\Lambda$ of $\mu+\nu$ must be a Fourier frame for $\rho: =\mu_1+\nu$. However, $\rho$ is a measure of $xy$-plane, so the projection $P(\Lambda)$ of $\Lambda$ onto the $xy$-plane must be a Fourier frame for $\rho$. Regarding $P(\Lambda)$ as a subset of ${\mathbb R}^2$ and $\rho$ as a measure on ${\mathbb R}^2$, we obtain a contradiction since $\rho$ is a sum of singular measure and an absolutely continuous measure of ${\mathbb R}^2$ and it does not admit any Fourier frames according to the pure-type principle.}

\medskip

\item{\rm We note also that $\mu$ $\nu$ are respectively  the measure ${\mathcal L}^2_{[0,1]^2\times \{0\}}$ and ${\mathcal L}^2_{\{0\}\times [0,1]^2}$, then $\mu+\nu$ admits a Fourier frame (\cite[Theorem 2.2]{L16}).  We can also see it easily as follows:}
$$
\mu+\nu = \left({\mathcal L}^1_{[0,1]\times\{0\}\times\{0\}}+{\mathcal L}^1_{\{0\}\times\{0\}\times[0,1]}\right) \times {\mathcal L}^1_{\{0\}\times[0,1]\times\{0\}}.
$$
{\rm All measures on the right admit a Fourier frame and the cross-product of two frame-spectral measures admits a Fourier frame, so $\mu+\nu$ is frame-spectral.}

 \end{enumerate}
 \end{example}

\section{Remarks}

We conclude with some remarks on the existence of Fourier frames for a sum of singular measures for further study. A  question  that arises naturally is that

\medskip

\noindent{\bf (Qu):} Suppose that $\mu$ and $\nu$ are frame-spectral measures on ${\mathbb R}^d$. When is the measure $\mu+\nu$ frame-spectral? Is translational absolute continuity of $\mu+\nu$  necessary?
\medskip

We  notice that the first two measures in Example \ref{example4} are translationally singular, while the third one is translationally absolutely continuous. Also, the measures in Theorem \ref{t-3-3} are translationally absolutely continuous if $\mu$ and $\nu$ are. It is plausible that translational absolute continuity of $\mu+\nu$ is necessary, which also confirms the conjecture \ref{conj1}. In \cite{L16}, Lev also suggested some examples such that $\mu+\nu$ is non-frame-spectral, while $\mu,\nu$ is frame-spectral.  Perhaps a specific example that can be asked is as follows:

%
%
%
%
%

\vspace{0.3cm}

\noindent{\bf {(Qu)}:} Let $\mu_4$ and $\mu_8$ denote the one-fourth Cantor measure and one-eighth Cantor measure respectively. Is $\mu_4 + \mu_8$ a frame-spectral measure and translationally absolutely continuous?

\vspace{0.3cm}

It looks like translational absolute continuity is related to the Hausdorff dimension of $K_{\mu_4}\cap (K_{\mu_8}+t)$. A recent study about the Hausdorff dimension of the intersection of translation of Cantor sets can be found in \cite{PP, FHR14} and the references therein. Concerning its frame-spectrality, Lev's paper \cite{L16} suggested a possible construction of frame spectrum of $\mu+\nu$, which requires to find a frame spectrum  for $\mu$ forming a Bessel sequence of $L^2(\nu)$ and vice versa for a frame spectrum of $\nu$. We do not know how these spectra can be found.  However, we have checked through computer programming putting orthonormal sequences of $\mu_8$ into $\mu_4$ as well as checking  the frame-spectral Cantor-type measure in \cite{LW} with their contraction scale very close to each other, we found all frame spectra of one of the measures appear to fail to be a Bessel sequence of the other. Thus, we conjecture that $\mu_4+\mu_8$ is not frame-spectral.

\medskip

 As the end of the paper, we provide the following example that demonstrates  a frame-spectral measure can be translational singular on the closed support, but it is translational absolutely continuous on some Borel supports, which shows that Borel support is necessary as in Remark \ref{remark} (3).

\begin{example}  \label{example5}
{\rm Let $A$ be a Borel set of positive Lebesgue measure  inside $[0,1]$ with the property that for any interval $I\subset [0,1]$, $0<{\mathcal L}^1(I\cap A)<{\mathcal L}^1(I)$ (such set exists and can be found in \cite{Rud}). Let also $B = [0,1]\setminus A$. Then $B$ also satisfies the property. Let $\Omega = A\cup (B+1)$ and $\mu = {\bf 1}_{\Omega}dx$. Then $\mu$ is a frame-spectral measure, $\mu$ is translationally absolutely continuous on $\Omega$, but is translationally singular on the closed support of $\mu$.  }
\end{example}

\begin{proof}
Since $\Omega$ is a fundamental domain of ${\mathbb Z}$, so $\Omega$ is a spectral set and $\mu$ is a spectral measure. It is also translationally absolutely continuous on $\Omega$. If we take $E\subset{\Omega}$ and $\mu(E)>0$, then if $\mu(F) =0$,
$$
\begin{aligned}
\omega_t(F) = \mu((E+t)\cap (F+t)) =& {\mathcal L}^1(((E\cap F)+t)\cap\Omega) \\
 \le&  {\mathcal L}^1((E\cap F)+t) = {\mathcal L}^1((E\cap F)) = \mu (E\cap F) =0.
\end{aligned}
$$
However, $K_{\mu} = [0,2]$ since $A, B$ are dense in $[0,1]$. If we take $E = [0,1]$ and $t=1$, then $\mu(B) = 0$, but
$$
\omega_1(B) = \mu((E\cap B)+1) = \mu(B+1) = {\mathcal L}^1(B)>0.
$$
 Thus, $\mu$  is translationally singular on the closed support of $\mu$.
\end{proof}

\medskip

\noindent{\bf Acknowledgement.} The authors would like thank Professor Nir Lev for suggesting a simpler proof of Theorem \ref{t-3-3} and providing Example \ref{example5} soon after the first version of the paper is released.

\end{document}